\documentclass[11pt]{article}

\usepackage{fullpage}
\usepackage[utf8]{inputenc}
\usepackage[T1]{fontenc}
\usepackage{amsmath,amsfonts,amssymb,amsthm,bbm}
\usepackage{stackrel}
\usepackage{enumerate}
\usepackage{graphicx}
\usepackage{hyperref}
\usepackage{subfigure}

\usepackage{latexsym} 
\usepackage{color} 

\usepackage{mathtools} 

\allowdisplaybreaks 

\newtheorem{theorem}{Theorem}

\newtheorem{lemma}[theorem]{Lemma}
\newtheorem{definition}[theorem]{Definition}

\newtheorem{remark}[theorem]{Remark}

\numberwithin{theorem}{section}
\numberwithin{equation}{section}
\numberwithin{figure}{section}

\newcommand{\ve}{\varepsilon}

\DeclareMathOperator{\diam}{diam}

\newcommand{\ZZ}{\mathbb{Z}}
\newcommand{\RR}{\mathbb{R}}
\newcommand{\CC}{\mathbb{C}}

\newcommand{\PP}{\mathbb{P}}
\newcommand{\EE}{\mathbb{E}}

\newcommand{\TT}{\mathbb{T}} 
\newcommand{\Ball}{B} 
\newcommand{\Ann}{A} 
\newcommand{\din}{\partial^{\textrm{in}}} 
\newcommand{\dout}{\partial^{\textrm{out}}} 
\newcommand{\cluster}{\mathcal{C}}
\newcommand{\Ch}{\mathcal{C}_H} 
\newcommand{\Cv}{\mathcal{C}_V} 
\newcommand{\circuit}{\mathcal{C}} 
\newcommand{\circuitevent}{\mathcal{O}} 
\newcommand{\colorseq}{\mathfrak{S}} 
\newcommand{\arm}{\mathcal{A}} 

\newcommand{\net}{\mathcal{N}} 

\newcommand{\Pdiam}{\PP_N^{\textrm{diam}}}
\newcommand{\Pvol}{\PP_N^{\textrm{vol}}}
\newcommand{\PPdiam}{\tilde{\PP}_N^{\textrm{diam}}}
\newcommand{\PPvol}{\tilde{\PP}_N^{\textrm{vol}}}

\newcommand{\lra}{\leftrightarrow}

\newcommand{\calE}{\mathcal{E}}
\newcommand{\calF}{\mathcal{F}}

\DeclareFontFamily{U}{mathx}{\hyphenchar\font45}
\DeclareFontShape{U}{mathx}{m}{n}{
      <5> <6> <7> <8> <9> <10>
      <10.95> <12> <14.4> <17.28> <20.74> <24.88>
      mathx10
      }{}
\DeclareSymbolFont{mathx}{U}{mathx}{m}{n}
\DeclareFontSubstitution{U}{mathx}{m}{n}
\DeclareMathAccent{\wideparen}{0}{mathx}{"75}

\begin{document}

\title{Boundary rules and breaking of self-organized criticality\\ in 2D frozen percolation}

\author{Jacob van den Berg\footnote{CWI and VU University Amsterdam}, Pierre Nolin\footnote{ETH Z\"urich}}

\date{}

\maketitle

\begin{abstract}
We study frozen percolation on the (planar) triangular lattice, where connected components stop growing (``freeze'') as soon as their ``size'' becomes at least $N$, for some parameter $N \geq 1$. The size of a connected component can be measured in several natural ways, and we consider the two particular cases of diameter and volume (i.e. number of sites).

Diameter-frozen and volume-frozen percolation have been studied in previous works (\cite{BLN2012, Ki2015} and \cite{BN2015, BKN2015}, resp.), and they display radically different behaviors. These works adopt the rule that the boundary of a frozen cluster stays vacant forever, and we investigate the influence of these ``boundary conditions'' in the present paper. We prove the (somewhat surprising) result that they strongly matter in the diameter case, and we discuss briefly the volume case.

\bigskip

\end{abstract}

\section{Introduction}

\subsection{Frozen percolation}

In statistical physics, the phenomenon of self-organized criticality (or SOC for short) refers, roughly speaking, to the spontaneous (approximate) arising of a critical regime without any fine-tuning of a parameter. Numerous works have been devoted to it, mostly in physics (see e.g. \cite{Ba1996, Je1998} and the references therein) but also on the rigorous mathematical side. The critical regime of independent percolation is of particular interest, and arises (or seems to arise in some sense) in models of forest fires \cite{DrSc1992, BB2006}, displacement of oil by water in a porous medium \cite{WW1983, DSV2009}, diffusion fronts \cite{SRG1985, No2009}, and in frozen percolation, the topic of the present paper. In all these processes, the stochastic system under consideration ``selects'' a window around the percolation threshold where the relevant macroscopic behavior takes place.

Frozen percolation is a percolation-type growth process introduced by Aldous \cite{Al2000}, inspired by sol-gel transitions \cite{St1943}. In \cite{Al2000}, it is shown that in the particular case of the binary tree, frozen percolation displays a striking exact form of SOC: at any time $p \geq p_c = \frac{1}{2}$, the finite (``non-frozen'') clusters have the same distribution as critical clusters, while the infinite (``frozen'') clusters all look like incipient infinite clusters. In two dimensions, it was shown in \cite{Ki2015} that diameter-frozen percolation also displays a form of SOC: all frozen clusters freeze in a near-critical window around $p_c$, and consequently, they all look similar to critical percolation clusters. Here, we prove the somewhat unexpected result that in two dimensions, the particular mechanism to freeze clusters (what we call ``boundary conditions'', or rather ``boundary rules'') matters strongly, and can lead to a partial breaking of SOC. As we explain below, this result is based on a rather subtle geometric argument showing the existence of narrow passages that can be used to create highly supercritical frozen clusters.

We now focus on a specific version of frozen percolation in two dimensions, defined in terms of site percolation on the triangular lattice $\TT$. This lattice has vertex set
$$V(\TT) = \big\{ x + y e^{i \pi / 3} \in \CC \: : \: x, y \in \ZZ \big\},$$
and its edge set $E(\TT)$ is obtained by connecting all vertices $v, v' \in V(\TT)$ at Euclidean distance $1$ apart (in this case, we say that $v$ and $v'$ are neighbors, and we denote it by $v \sim v'$). For a subset $A \subseteq V(\TT)$, we consider the following two ways of measuring its ``size''. We call \emph{diameter of $A$}, denoted by $\diam(A)$, its diameter for the $L^{\infty}$ norm $\|.\| := \|.\|_{\infty}$ (where $A$ is seen as a subset of $\CC \simeq \RR^2$): $\diam(A) = \sup_{v,v' \in A} \| v - v' \|$. On the other hand, the \emph{volume of $A$}, denoted by $|A|$, is simply the number of vertices that it contains.

Let us consider a family $(\tau_v)_{v \in V(\TT)}$ of i.i.d. random variables uniformly distributed on $[0,1]$. For each $p \in [0,1]$, we declare a vertex $v$ to be $p$-black (resp. $p$-white) if $\tau_v \leq p$ (resp. $\tau_v > p$). Then, $p$-black and $p$-white vertices are distributed according to independent site percolation with parameter $p$: vertices are black or white with probability $p$ and $1-p$, respectively, independently of each other. In the following, the corresponding probability measure is denoted by $\PP_p$, while we use the notation $\PP$ for events involving the whole collection of random variables $(\tau_v)_{v \in V(\TT)}$. It is now classical \cite{Ke1980} that site percolation on $\TT$ displays a phase transition at the critical parameter $p_c = \frac{1}{2}$: for all $p \leq p_c$, there is a.s. no infinite $p$-black cluster, while for $p > p_c$, there exists a.s. an infinite $p$-black cluster, which is moreover unique. For an introduction to percolation theory, the reader can consult \cite{Gr1999}.

The diameter- and volume-frozen percolation processes are defined in terms of the same family $(\tau_v)_{v \in V(\TT)}$. These processes have a parameter $N \geq 1$. At time $t=0$, we start with the initial configuration where all the vertices in $V(\TT)$ are white, and as time $t$ increases from $0$ to $1$, each vertex $v$ can become black only at time $t = \tau_v$, \emph{iff} all the black clusters adjacent to $v$ have a diameter (resp. volume) $< N$. Note that if $v$ is not allowed to turn black at time $\tau_v$, then it stays white until time $t=1$. Hence, black clusters grow until their diameter (resp. volume) becomes $\geq N$, and then their growth is stopped. In this case, the cluster (and the vertices that it contains) is said to be \emph{frozen}. When referring to this process, we use the notation $\Pdiam$ (resp. $\Pvol$). As noted in \cite{BLN2012}, this process is well-defined since it can be represented as a finite-range interacting particle system (see \cite{Li2005} for general theory).

These processes were studied in the previous works \cite{BLN2012, Ki2015} (diameter-frozen percolation) and \cite{BN2015, BKN2015} (volume-frozen percolation). With this definition, a cluster freezes when it becomes large, and all the vertices along its outer boundary then stay white until the end. However, one may ask whether for all applications these ``boundary conditions'' are always the most natural, and if tweaking them would lead to a different macroscopic behavior. This leads us to discuss \emph{modified} (diameter- and volume-) frozen percolation processes, where, informally speaking, the sites adjacent to a frozen cluster become black (and may freeze) at a later time.

More precisely, these processes are defined as follows. Again, we use the collection of random variables $(\tau_v)_{v \in V(\TT)}$, and we start with all vertices white. Since a frozen cluster can touch black clusters, we have to choose a slightly different representation of the process. Now, a vertex $v \in V(\TT)$ can be in three possible states: either white, black (unfrozen), or \emph{frozen}. As time $t$ increases, each vertex $v$ changes state at time $t = \tau_v$. Just before this time, it is white, and it then becomes either black or frozen, depending on the configuration around it: let $\mathcal{B}_{t^-}(v) \subseteq V(\TT)$ be the union of $v$ and all the black clusters adjacent to $v$ at time $t^-$. If $\mathcal{B}_{t^-}(v)$ has a diameter (resp. volume) $\geq N$, then all the vertices in $\mathcal{B}_{t^-}(v)$ change state, from black to frozen. Otherwise, $v$ just becomes black (and may become frozen at a later time). These modified processes are denoted by $\PPdiam$ and $\PPvol$, and it is not difficult to see that they are well-defined, since they can also be seen as finite-range interacting particle systems.

\begin{figure}
\begin{center}

\subfigure{\includegraphics[width=.44\textwidth]{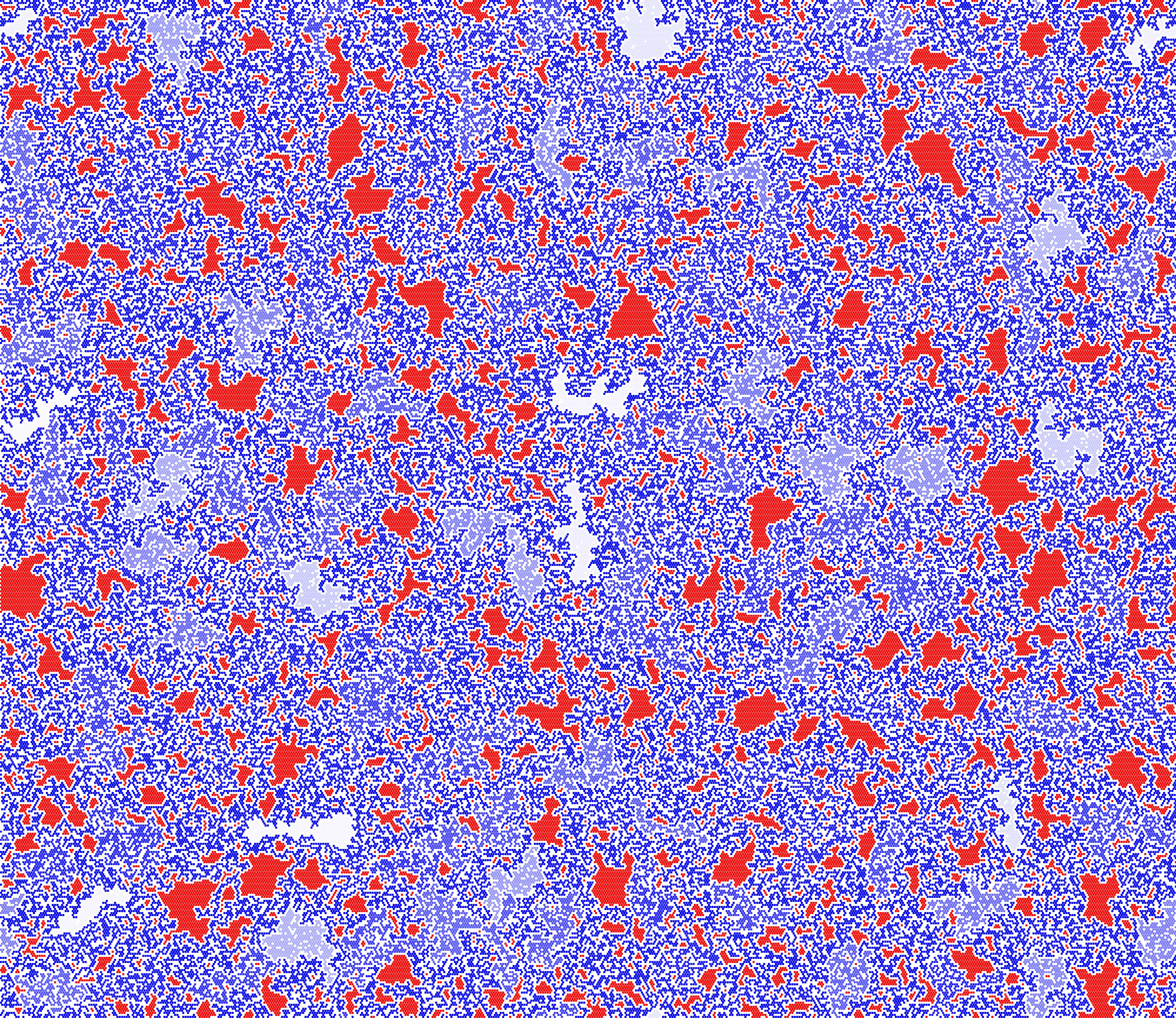}}
\hspace{0.08\textwidth}
\subfigure{\includegraphics[width=.44\textwidth]{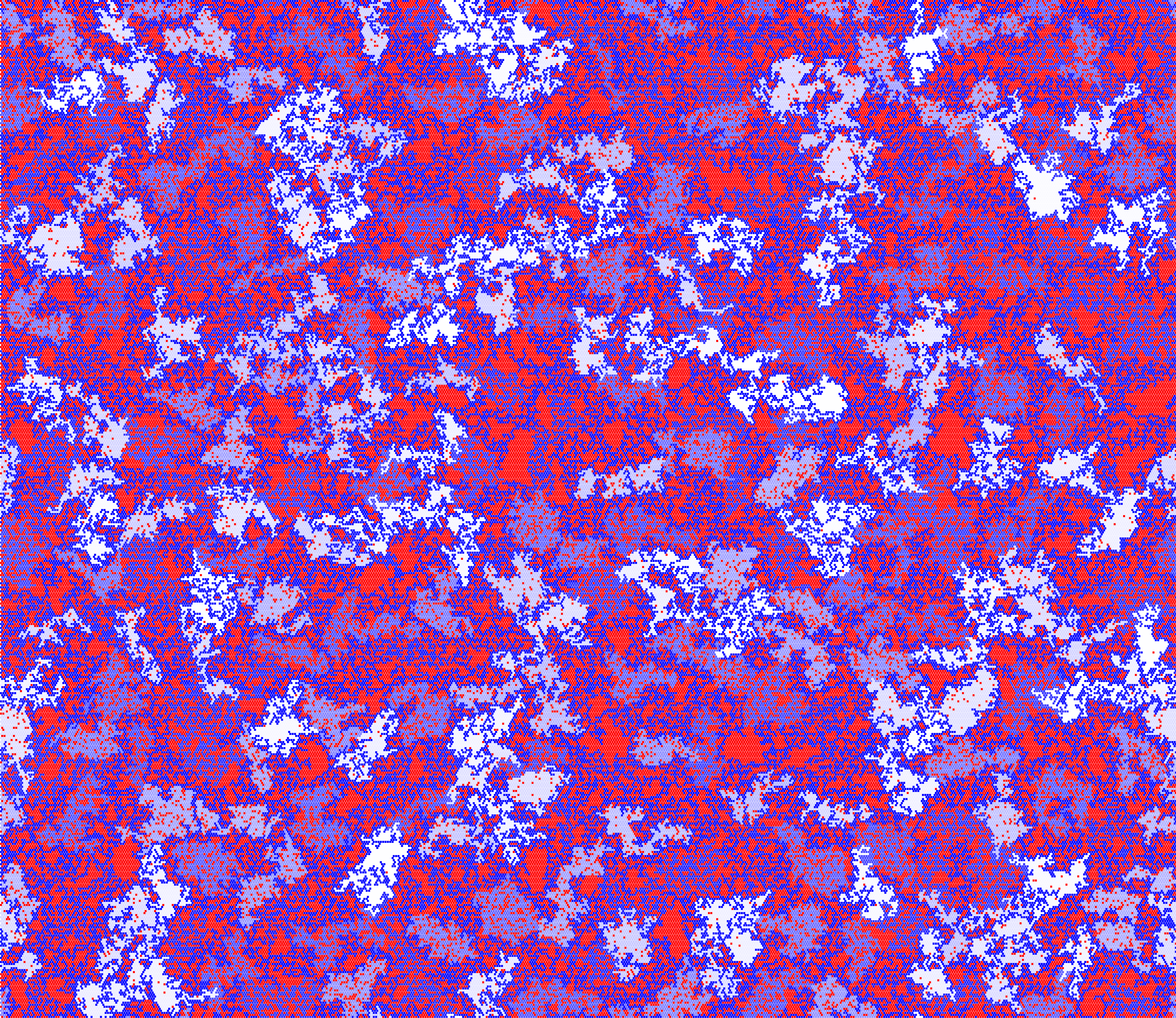}}\\

\caption{\label{fig:final_config} Final configuration for diameter-frozen percolation with parameter $N=30$ (Fig. Demeter Kiss): ``original'' process (left) and ``modified'' process (right). The blue sites are occupied and frozen (where a lighter blue corresponds to a later time of freezing), the red sites are occupied and non-frozen (i.e. trapped in a hole with diameter $< N$), and the white sites (only present in the original process) are vacant (they lie along the boundary of frozen clusters).}
\end{center}
\end{figure}


\subsection{Effect of boundary rules} \label{sec:intro_bc}

In the case of diameter-frozen percolation, we show that boundary conditions do have a strong effect. We first discuss briefly the results of \cite{Ki2015} for the original process, i.e. when the vertices along the boundary of a frozen cluster stay white forever. In that paper, it is proved that frozen clusters only arise in a near-critical window around $p_c$: for every fixed $K > 1$,
\begin{equation}
\liminf_{N \to \infty} \Pdiam \big( \text{some vertex in } [-KN, KN]^2 \text{ freezes outside } [p_{-\lambda}(N), p_{\lambda}(N)] \big) \stackrel{\lambda \to \infty}{\longrightarrow} 0,
\end{equation}
where $p_{\lambda}(N) \simeq p_c + \frac{\lambda}{N^{3/4}}$ refers to the usual near-critical parameter scale (a precise definition requires the introduction of more percolation notation, and is postponed to Section \ref{sec:def_nc_parameter}). Also, macroscopic non-frozen clusters asymptotically have full density:
\begin{equation} \label{eq:macro_non_frozen_orig}
\liminf_{N \to \infty} \Pdiam \big( \diam ( \cluster_1(0) ) \in [\ve N, (1 - \ve) N] \big) \stackrel{\ve \to 0^+}{\longrightarrow} 1,
\end{equation}
where $\cluster_1(0)$ denotes the black cluster of $0$ at time $1$ (which we consider to be $\emptyset$ if $0$ is not black). In particular, $\Pdiam ( 0 \text{ freezes}) \to 0$ as $N \to \infty$.

In contrast, we prove the following result for the modified process. If we allow the white vertices along the boundary of a frozen cluster to become black later (and to possibly freeze), then some very dense frozen clusters form at a late time (close to $1$).

\begin{theorem} \label{thm:bc_diam}
For the modified diameter-frozen percolation process on $\TT$,
\begin{equation}
\liminf_{N \to \infty} \PPdiam ( 0 \text{ freezes} ) > 0.
\end{equation}
\end{theorem}

\begin{remark} ~
\begin{itemize}
\item[(a)] Actually, the following more precise property holds: for all $\ve > 0$,
\begin{equation}
\liminf_{N \to \infty} \PPdiam \big( 0 \text{ freezes in } \big( 1 - N^{-\frac{3}{4} + \ve}, 1 \big) \big) > 0.
\end{equation}

\item[(b)] The construction used in the proof of Theorem \ref{thm:bc_diam} also provides more information about the final configuration. In the scenario that we give, and which occurs with a probability bounded away from $0$ (as $N \to \infty$), $0$ lies in a macroscopic ``chamber'' (with diameter smaller than $N$ but of order $N$) which is ``protected from the outside'' until time $1 - N^{-\frac{3}{4} + \ve}$. In that scenario, $0$ lies in a highly supercritical cluster which freezes at some time $p^* \in \big( 1 - N^{-\frac{3}{4} + \ve}, 1 \big)$. In particular, our proof implies that for every $\ve > 0$, there is a probability bounded away from $0$ that all the sites within distance $N^{\frac{3}{8} - \ve}$ from $0$ are frozen.
\end{itemize}
\end{remark}

In the last section we discuss, briefly and informally, the influence of boundary conditions for volume-frozen percolation. Roughly speaking and contrary to what happens in the diameter case, we do not expect the choice of boundary conditions to have a considerable effect in the volume case. In some sense, this case seems to be more ``robust''.


\subsection{Organization of the paper}

We first discuss independent percolation in Section \ref{sec:percolation}. After fixing notations, we collect tools from critical and near-critical percolation which are central in our proofs. We then study the modified diameter-frozen percolation process in Section \ref{sec:diam_frozen}, where we prove Theorem \ref{thm:bc_diam}. For that, we use an ``ad-hoc'' configuration of near-critical clusters (see Figure \ref{fig:two_chambers} below). We prove that such a configuration occurs with reasonable probability, and (combined with some extra features, see Figure \ref{fig:regularity_boundary}) gives a scenario where $0$ freezes. Finally, we make a few comments on modified volume-frozen percolation in Section \ref{sec:vol_frozen}.

\section{Preliminary: independent percolation} \label{sec:percolation}

\subsection{Setting and notations}

In this section, we first fix some notations regarding site percolation on $\TT$. For a subset $A \subseteq V(\TT)$, we denote by $\din A := \{v \in A \: : \: v \sim v'$ for some $v' \in A^c\}$ its inner boundary, and by $\dout A := \din (A^c)$ its outer boundary.

A path of length $k$ ($k \geq 1$) is a sequence of vertices $v_0 \sim v_1 \sim \ldots \sim v_k$. Two vertices $v, v' \in V(\TT)$ are said to be connected if for some $k \geq 1$, there exists a path of length $k$ from $v$ ot $v'$ (i.e. such that $v_0 = v$ and $v_k = v'$) containing only black sites: we denote this event by $v \lra v'$. More generally, two subsets $A, A' \subseteq V(\TT)$ are said to be connected if there exist $v \in A$ and $v' \in A'$ such that $v \lra v'$, which we denote by $A \lra A'$. Note that we sometimes consider white paths: in this case, the color is always specified explicitly, and we use the notation $\lra^*$.

The event that a vertex $v$ belongs to an infinite black cluster is denoted by $v \lra \infty$, and we write
\begin{equation} \label{eq:def_theta}
\theta(p) := \PP_p(0 \lra \infty).
\end{equation}
For a rectangle $R = [x_1,x_2] \times [y_1,y_2]$, a black path connecting the left and right (resp. top and bottom) sides is called a horizontal (resp. vertical) crossing, and the event that such a path exists is denoted by $\Ch(R)$ (resp. $\Cv(R)$). We also write $\Ch^*(R)$ and $\Cv^*(R)$ for the similar events with white paths.

Let $\Ball_n := [-n,n]^2$ be the ball of radius $n > 0$ around $0$ for $\|.\|$, and for $0 < m < n$, let $\Ann_{m,n} := \Ball_n \setminus \Ball_m$. For $z \in \CC$, we write $\Ball_n(z) := z + \Ball_n$, and $\Ann_{m,n}(z) := z + \Ann_{m,n}$. For such an annulus $A = \Ann_{m,n}(z)$, we denote by $\circuitevent(A)$ (resp. $\circuitevent^*(A)$) the event that there exists a black (resp. white) circuit in $A$ surrounding $\Ball_m(z)$. If $k \geq 1$ and $\sigma \in \colorseq_k := \{b,w\}^k$ (where we write $b$ for black, and $w$ for white), we also define the event $\arm_{\sigma}(A)$ that there exist $k$ disjoint paths $(\gamma_i)_{1 \leq i \leq k}$ in $A$ with respective colors $\sigma_i$, in counter-clockwise order, each ``crossing'' $A$ (i.e. connecting $\din \Ball_n(z)$ and $\dout \Ball_m(z)$). We also use the notations
\begin{equation} \label{eq:def_pi}
\arm_{\sigma}(m,n) := \arm_{\sigma}(\Ann_{m,n}) \quad \text{and} \quad \pi_{\sigma}(m,n) := \PP_{p_c}\big( \arm_{\sigma}(m,n) \big),
\end{equation}
and we simply write $\pi_{\sigma}(n) := \pi_{\sigma}(0,n)$ when $m = 0$. For $k \geq 1$, we also use the shorthand notations $\arm_k$ and $\pi_k$ in the particular case when $\sigma =(bwb\ldots) \in \colorseq_k$ is alternating (i.e the color sequence ends with $\sigma_{k} = b$ or $w$ according to the parity of $k$).

We will use repeatedly the usual Harris inequality for monotone events, and in some cases, we will need the slightly more general version below (see Lemma 3 in \cite{Ke1987}), for ``locally monotone'' events.
\begin{lemma} \label{lem:Harris_gen}
Consider $\calE^+$, $\tilde{\calE}^+$ two increasing events, $\calE^-$, $\tilde{\calE}^-$ two decreasing events, and assume that for some disjoint subsets $A, A^+, A^- \subseteq V(\TT)$, these events depend only on the sites in $A \cup A^+$, $A^+$, $A \cup A^-$, and $A^-$, respectively. Then
$$\PP \big( \tilde{\calE}^+ \cap \tilde{\calE}^- \cap \calE^+ \cap \calE^- \big) \geq \PP \big( \tilde{\calE}^+ \big) \PP \big( \tilde{\calE}^- \big) \PP \big( \calE^+ \cap \calE^- \big).$$
\end{lemma}
This result follows easily by first conditioning on the configuration in $A$, and then applying twice the Harris inequality (to the configuration in $A^+$ and in $A^-$).

\subsection{Critical and near-critical percolation}

Our results are based on a precise description of the behavior of percolation through its phase transition, i.e. at and near criticality. We now collect classical properties of near-critical percolation which are used throughout the proofs. We define the characteristic length $L$ by: for $p < p_c = \frac{1}{2}$,
\begin{equation} \label{eq:def_L}
L(p) = \min \big\{ n > 0 \: : \: \PP_p \big( \Cv( [0,2n] \times [0,n] ) \big) \leq 0.01 \big\},
\end{equation}
and $L(p) = L(1-p)$ for $p > p_c$. We also set $L(p_c) = \infty$.

\begin{itemize}
\item[(i)] \emph{Russo-Seymour-Welsh (RSW) bounds.} For all $k \geq 1$, there exists a universal constant $\delta_k > 0$ such that: for all $p \in (0,1)$, and $n \leq L(p)$,
\begin{equation} \label{eq:RSW}
\PP_p \big( \Ch( [0,kn] \times [0,n] ) \big) \geq \delta_k \quad \text{and} \quad \PP_p \big( \Ch^*( [0,kn] \times [0,n] ) \big) \geq \delta_k.
\end{equation}

\item[(ii)] \emph{Exponential decay with respect to L(p).} There exist universal constants $c_i, c'_i > 0$ ($i \in \{1,2\}$) such that: for all $p < p_c$, and $n \geq 1$,
\begin{equation} \label{eq:exp_decay}
\PP_p \big( \Cv( [0,2n] \times [0,n] ) \big) \leq c_1 e^{- c_2 \frac{n}{L(p)}} \quad \text{and} \quad \PP_p \big( \Ch( [0,2n] \times [0,n] ) \big) \geq c'_1 e^{- c'_2 \frac{n}{L(p)}}
\end{equation}
(see Lemma 39 in \cite{No2008}).

\item[(iii)] \emph{A-priori bounds on arm events.} There exist universal constants $c, \beta > 0$ such that: for all $p \in (0,1)$, and $0 < m < n \leq L(p)$,
\begin{equation} \label{eq:1arm}
\PP_p \big( \arm_1(m,n) \big) \leq c \bigg( \frac{m}{n} \bigg)^{\beta}.
\end{equation}
It follows from this bound, the ``universal'' arm exponent for $\arm_5$ which equals $2$ (see e.g. Theorem 24 (3) in \cite{No2008}), and Reimer's inequality, that for some universal constants $c', c'' > 0$: for all $p \in (0,1)$, and $0 < m < n \leq L(p)$,
\begin{equation} \label{eq:46arms}
\PP_p \big( \arm_4(m,n) \big) \geq c' \bigg( \frac{m}{n} \bigg)^{2 - \beta} \quad \text{and} \quad \PP_p \big( \arm_6(m,n) \big) \leq c'' \bigg( \frac{m}{n} \bigg)^{2 + \beta}.
\end{equation}

\item[(iv)] \emph{Asymptotic equivalences.} For the functions $\theta$ and $L$, the following estimates hold:
\begin{equation} \label{eq:equiv_theta}
\theta(p) \asymp \pi_1(L(p)) \quad \text{as $p \searrow p_c$}
\end{equation}
(see Theorem 2 in \cite{Ke1987}, or (7.25) in \cite{No2008}), and
\begin{equation} \label{eq:equiv_L}
\big| p - p_c \big| L(p)^2 \pi_4 \big( L(p) \big) \asymp 1 \quad \text{as $p \to p_c$}
\end{equation}
(see (4.5) in \cite{Ke1987}, or Proposition 34 in \cite{No2008}).
\end{itemize}

For site percolation on the triangular lattice $\TT$, conformal invariance at criticality (in the scaling limit) was proved by Smirnov \cite{Sm2001}. This property can be used to describe the scaling limit of critical percolation in terms of the Schramm-Loewner Evolution (SLE) processes (with parameter $6$), introduced in \cite{Sc2000}. It then allows one (using also \cite{LSW_2001a, LSW_2001b}) to derive arm exponents \cite{LSW2002, SW2001}.

\begin{itemize}
\item[(v)] \emph{Arm events at criticality.} For all $k \geq 1$, and $\sigma \in \colorseq_k$, there exists $\alpha_{\sigma} > 0$ such that
\begin{equation} \label{eq:arm_exponent}
\pi_{\sigma}(k,n) = n^{- \alpha_{\sigma} + o(1)} \quad \text{as $n \to \infty$.}
\end{equation}
Moreover, the value of $\alpha_{\sigma}$ is known, except in the monochromatic case (for $k \geq 2$ arms of the same color).
\begin{itemize}
\item For $k=1$, $\alpha_{\sigma} = \frac{5}{48}$.

\item For all $k \geq 2$, and $\sigma \in \colorseq_k$ containing both colors, $\alpha_{\sigma} = \frac{k^2-1}{12}$.
\end{itemize}
\end{itemize}

\subsection{Additional results}

In our proofs, we also make use two more specific results that we now state.

First, we need a ``separation'' property for $4$-arm events. For $n \geq 1$ and $p, p' \in (0,1)$, we denote by $\tilde{\arm}_4^{p,p'}(n)$ the event that there exist four paths $\gamma_i$ ($1 \leq i \leq 4$), in counterclockwise order, connecting $\partial 0$ to the right, top, left, and bottom sides of $\Ball_{3n}$, respectively, and such that
\begin{itemize}
\item $\gamma_1$ and $\gamma_3$ are $p$-black and stay in (resp.) $[-n,3n] \times [-n,n]$ and $[-3n,n] \times [-n,n]$,
\item $\gamma_2$ and $\gamma_4$ are $p'$-white and stay in (resp.) $[-n,n] \times [-n,3n]$ and $[-n,n] \times [-3n,n]$.
\end{itemize}
Then, it follows from classical results about near-critical percolation that there exist universal constants $c_1, c_2 > 0$ such that: for all $p,p' \in (0,1)$, and $1 \leq n \leq \min(L(p),L(p'))$,
\begin{equation} \label{eq:sep_4arms}
c_1 \pi_4(n) \leq \PP \big( \tilde{\arm}_4^{p,p'}(n) \big) \leq c_2 \pi_4(n).
\end{equation}

We also make use of the following geometric construction.

\begin{definition} \label{def:net}
For $1 \leq m \leq n$, we consider all the horizontal and vertical rectangles of the form
$$\Ball_m(2mx) \cup \Ball_m(2mx'), \quad \text{with } x, x' \in \Ball_{\lceil n / 2m \rceil + 1}, \: x \sim x'$$
(covering the ball $\Ball_{n+2m}$), and for $p \in (0,1)$, we denote by $\net_p(m,n)$ the event that in each of these rectangles, there exists a $p$-black crossing in the long direction.
\end{definition}

Note that the event $\net_p(m,n)$ implies the existence of a $p$-black cluster $\net$ such that all the $p$-black clusters and all the $p$-white clusters that intersect $\Ball_n$, except $\net$ itself, have a diameter at most $4m$. In the following, such a cluster $\net$ is called a \emph{net with mesh $m$}.

\begin{lemma} \label{lem:net}
There exist universal constants $c_1, c_2 > 0$ such that: for all $1 \leq m \leq n$ and $p > p_c$,
\begin{equation} \label{eq:net}
\PP \big( \net_p(m,n) \big) \geq 1 - c_1 \Big( \frac{n}{m} \Big)^2 e^{-c_2 \frac{m}{L(p)}}.
\end{equation}
\end{lemma}

\begin{proof}[Proof of Lemma \ref{lem:net}]
This is an immediate consequence of the exponential decay property \eqref{eq:exp_decay}, since the definition of $\net_p(m,n)$ involves of order $\big( \frac{n}{m} \big)^2$ rectangles, each with side lengths $4m$ and $2m$.
\end{proof}

\subsection{Near-critical parameter scale} \label{sec:def_nc_parameter}

For the constructions in Section \ref{sec:diam_frozen}, the following near-critical parameter scale (already mentioned in Section \ref{sec:intro_bc}) is convenient to work with.
\begin{definition}
For $\lambda \in \RR$ and $N \geq 1$, let
\begin{equation} \label{eq:def_nc_scale}
p_{\lambda}(N) := p_c + \frac{\lambda}{N^2 \pi_4(N)}.
\end{equation}
\end{definition}
This particular choice has turned out to be quite suitable to study near-critical percolation and related phenomena, see e.g. \cite{NW2009, GPS2013a, Ki2015}. Note that for every fixed $\lambda$, $p_{\lambda}(N) \to p_c$ as $N \to \infty$ (using the a-priori lower bound on $4$-arm events \eqref{eq:46arms}). In particular, $p_{\lambda}(N) \in (0,1)$ for $N$ large enough. We use the following properties.
\begin{itemize}
\item[(i)] For every fixed $\lambda \in \RR$,
\begin{equation}
L \big( p_{\lambda}(N) \big) \asymp N \quad \text{as $N \to \infty$.}
\end{equation}

\item[(ii)] On the other hand,
\begin{equation} \label{eq:prop2_nc}
\limsup_{N \to \infty} \frac{L \big( p_{\lambda}(N) \big)}{N} \longrightarrow 0 \quad \text{as $\lambda \to \pm \infty$.}
\end{equation}

\item[(iii)] \emph{RSW bounds.} For all $\lambda \geq 0$ and $k \geq 1$, there exists a constant $\bar{\delta}_k = \bar{\delta}_k(\lambda) > 0$ such that: for all $N \geq 1$, $n \leq N$, and $p \in \big[ p_{-\lambda}(N), p_{\lambda}(N) \big]$,
\begin{equation} \label{eq:RSW_nc}
\PP_p \big( \Ch( [0,kn] \times [0,n] ) \big) \geq \bar{\delta}_k \quad \text{and} \quad \PP_p \big( \Ch^*( [0,kn] \times [0,n] ) \big) \geq \bar{\delta}_k.
\end{equation}
\end{itemize}

Properties (i) and (ii) follow from the definition of $p_{\lambda}$ \eqref{eq:def_nc_scale}, the asymptotic equivalence for $L$ \eqref{eq:equiv_L}, and \eqref{eq:46arms} for $4$ arms. Property (iii) then follows from (i) and (multiple use of) \eqref{eq:RSW}.

\section{Diameter-frozen percolation} \label{sec:diam_frozen}

We now turn our attention to diameter-frozen percolation. We recall in Section \ref{sec:diam_frozen_chamber} the proof of the main result in \cite{BLN2012}, which also applies to the modified diameter-frozen percolation process. This proof uses a construction which is the starting point of the more complicated construction in the proof of Theorem \ref{thm:bc_diam2}, in Section \ref{sec:diam_frozen_proof}. Finally, we discuss some open questions in Section \ref{sec:diam_frozen_open}.

\subsection{Construction of macroscopic chambers} \label{sec:diam_frozen_chamber}

First, we note that the main construction from \cite{BLN2012} (see Theorem 1.1 and Figure 1 in that paper) works in exactly the same way for the process with modified boundary conditions. Hence, we obtain the analog of the main result in \cite{BLN2012}: for all $0 < a < b < 1$,
\begin{equation} \label{eq:macro_non_frozen}
\liminf_{N \to \infty} \PPdiam \big( \diam ( \cluster_1(0) ) \in [aN, bN] \big) > 0.
\end{equation}
In other words, macroscopic non-frozen clusters (with diameter $< N$ but of order $N$) have a positive density.

Since we are using this construction as a building block for the proof of Theorem \ref{thm:bc_diam}, we remind it quickly to the reader on Figure \ref{fig:one_chamber}. For that, we choose $\eta, \ell \in (0, 1)$ such that
\begin{itemize}
\item $a + 6 \eta \leq b$ (so that the inner chamber has a diameter between $aN$ and $bN$),

\item $a + 7 \eta < 1$ and $\ell + 4 \eta < 1$ (so that $\circuit$ and the crossing in $r_1$ cannot freeze separately),

\item and $\ell + a + 4 \eta > 1$ (so that the big structure, that contains both $\circuit$ and part of the crossing, freezes before time $p_c$). Note that the $p_{-\lambda}(N)$-white crossing of $r'_1$ prevents that part of $\circuit$ freezes already with the crossing of $r_1$ before every site of $\circuit$ is black.
\end{itemize}
We can choose for instance $\ell = 1 - a$, and then $\eta > 0$ small enough. This construction creates a cluster which freezes at some time in $[p_{-\lambda}(N), p_c]$, and completely surrounds $\Ball_{\frac{a}{2} N}$ (without intersecting it). In this ``chamber'' with diameter $<N$, no connected component can freeze and we just observe an independent percolation configuration. In particular, all the sites are black at time $1$, which produces a non-frozen cluster with a diameter between $aN$ and $bN$.

\begin{figure}
\begin{center}

\includegraphics[width=.75\textwidth]{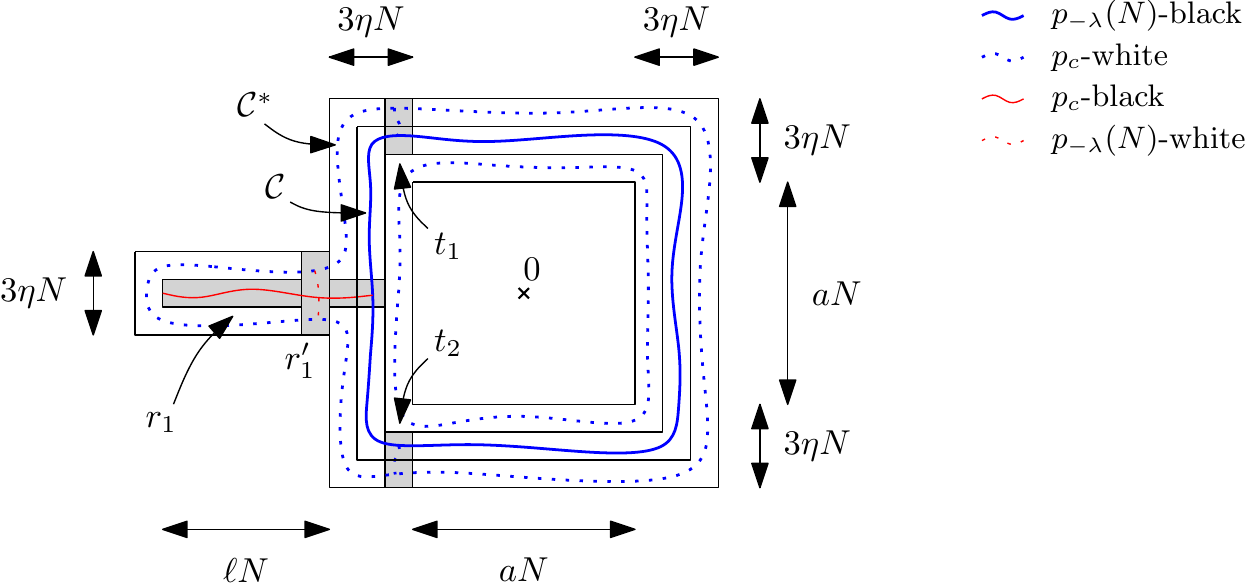}
\caption{\label{fig:one_chamber} Construction used to prove \eqref{eq:macro_non_frozen}, where all the ``corridors'' have width $\eta N$. When the big structure (containing $\circuit$ and part of the crossing in $r_1$) freezes, it leaves a hole whose boundary lies in $\Ann_{\frac{a}{2} N,(\frac{a}{2} + 3 \eta) N}$. Note that the $p_c$-white crossings in $t_1$ and $t_2$ prevent the appearance of big clusters other than the ones that we want to be created.}

\end{center}
\end{figure}

\subsection{Existence of highly supercritical frozen clusters} \label{sec:diam_frozen_proof}

We now prove Theorem \ref{thm:bc_diam} about the appearance of clusters freezing at a time very close to $1$. We actually prove the (slightly more precise) result below. Theorem \ref{thm:bc_diam} easily follows from it by plugging in the value $\alpha_4 = \frac{5}{4}$ of the $4$-arm exponent (see \eqref{eq:arm_exponent} and the paragraph below).

\begin{theorem} \label{thm:bc_diam2}
Consider the modified diameter-frozen percolation process on $\TT$. For every $\ve > 0$,
\begin{equation} \label{eq:bc_diam2}
\liminf_{N \to \infty} \PPdiam \bigg( 0 \text{ freezes in } \bigg( 1 - \frac{\ve}{N^2 \pi_4(N)}, 1 \bigg) \bigg) > 0.
\end{equation}
\end{theorem}

\subsubsection{Passage sites}

We now introduce an event which is instrumental in the proof of Theorem \ref{thm:bc_diam2}.

\begin{definition}
Let $n \geq 1$, $\delta>0$, and $p \in (0, p_c]$. We define the event $\Gamma_p^{n,\delta}$, depending on the sites in the box $\Ball_{3n}$, that there exists a vertical crossing $\gamma$ of $[-n,n] \times [-3n,3n]$ with the following two properties (see Figure \ref{fig:def_gamma}).
\begin{itemize}
\item[(i)] $\gamma$ is $p_c$-white.
\item[(ii)] There are at least $\delta n^2 \pi_4(n)$ sites $v \in \Ball_n$ along $\gamma$ which are ``passage sites'': each such $v$ possesses two neighbors $v_1$ and $v_2$ which are connected in $[-3n,3n] \times [-n,n]$ by $p$-black paths to the left and right sides of $\Ball_{3n}$, respectively.
\end{itemize}
\end{definition}

\begin{figure}
\begin{center}

\includegraphics[width=.5\textwidth]{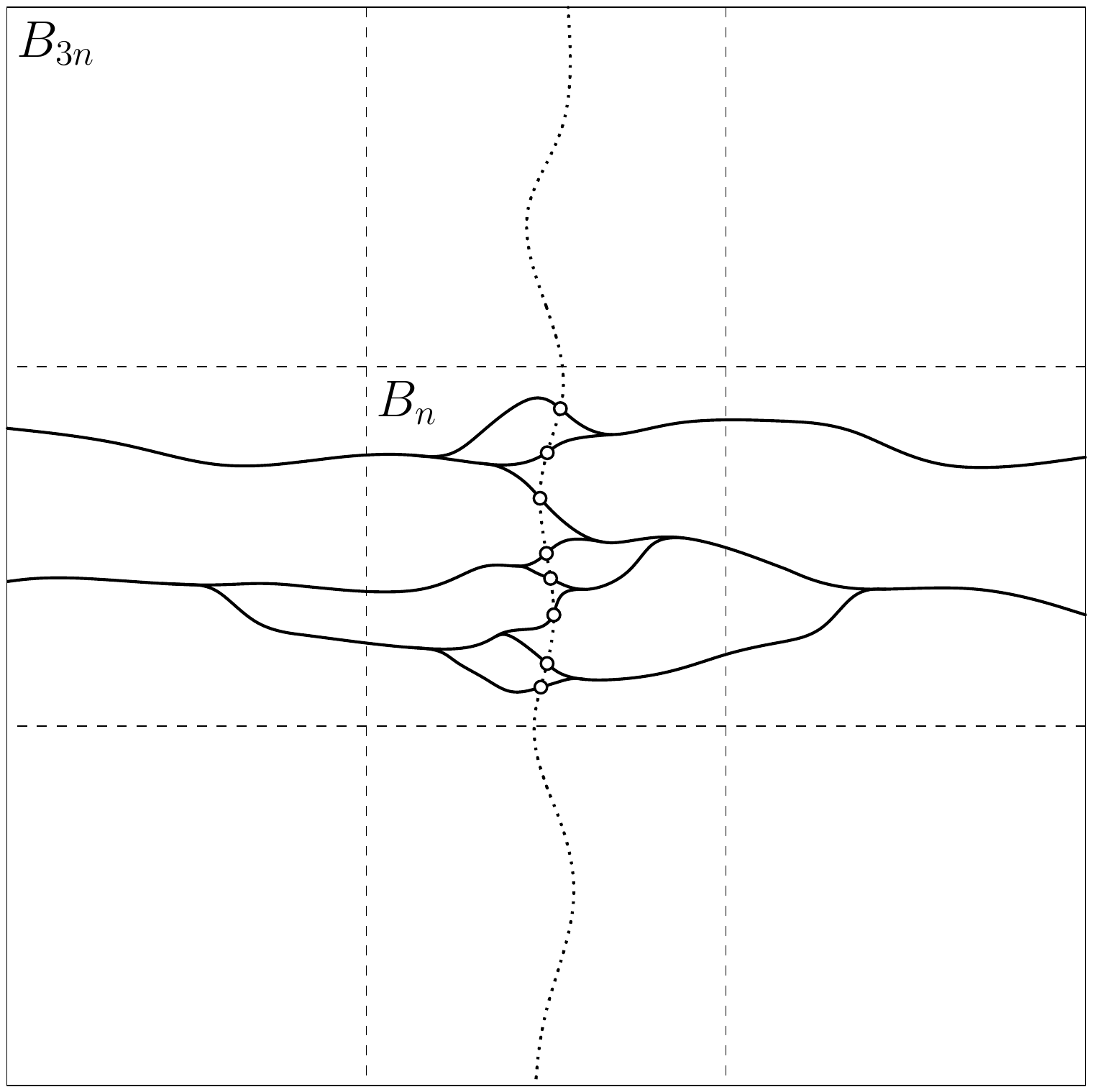}
\caption{\label{fig:def_gamma} This figure depicts the event $\Gamma_p^{n,\delta}$ used in the proof of Theorem \ref{thm:bc_diam2}. The solid paths are $p$-black, the dotted ones are $p_c$-white, and there are at least $\delta n^2 \pi_4(n)$ ``passage sites'', i.e. $p_c$-white vertices with neighbors connected by $p$-black paths to the left and right sides of $\Ball_{3n}$.}

\end{center}
\end{figure}

\begin{lemma} \label{lem:passage_sites}
For every $\lambda \geq 0$, there exists $\delta = \delta(\lambda) > 0$ such that: for all $N \geq 1$ and $n \leq N$, 
\begin{equation}
\PP \big( \Gamma_{p_{-\lambda}(N)}^{n,\delta} \big) \geq \delta.
\end{equation}
\end{lemma}

\begin{proof}[Proof of Lemma \ref{lem:passage_sites}]
Consider $\lambda \geq 0$. For $1 \leq n \leq N$ and $v \in \Ball_n$, let $\bar{\arm}_4(n,v) = \bar{\arm}_4^{\lambda,N}(n,v)$ be the event that
\begin{itemize}
\item[(i)] $v$ is $p_c$-white,

\item[(ii)] and there exist four paths $\gamma_i$ ($1 \leq i \leq 4$), in counterclockwise order, connecting $\partial v$ to the right, top, left, and bottom sides of $\Ball_{3n}$, respectively, and such that
\begin{itemize}
\item $\gamma_1$ and $\gamma_3$ are $p_{-\lambda}(N)$-black and stay in $[-3n,3n] \times [-n,n]$,
\item $\gamma_2$ and $\gamma_4$ are $p_c$-white and stay in $[-n,n] \times [-3n,3n]$.
\end{itemize}
\end{itemize}
By combining \eqref{eq:sep_4arms} with \eqref{eq:RSW} (to ``extend'' the arms suitably), we can obtain that uniformly for $n \leq N$ and $v \in \Ball_{n/2}$,
$$\PP(\bar{\arm}_4(n,v)) \asymp \pi_4(n)$$
(i.e. the constants in this asymptotic equivalence only depend on $\lambda$). In particular, if we introduce $X_n := \big| \{ v \in \Ball_{n/2} \: : \: \bar{\arm}_4(n,v)$ holds$\} \big|$, we have
$$\EE[X_n] \asymp n^2 \pi_4(n).$$
On the other hand, a standard summation argument for $4$-arm events (with the help of the quasi-multiplicativity property, see e.g. Proposition 17 in \cite{No2008}) yields
$$\EE[X_n^2] \leq c (n^2 \pi_4(n))^2$$
for some $c = c(\lambda) > 0$ (the a-priori lower bound for $4$ arms \eqref{eq:46arms} is also used). Hence, we can apply a second-moment argument to $X_n$: we obtain that for some $\delta = \delta(\lambda) > 0$,
$$\PP \big( X_n \geq \delta n^2 \pi_4(n) \big) \geq \delta,$$
which completes the proof of Lemma \ref{lem:passage_sites}.
\end{proof}

\subsubsection{Proof of Theorem \ref{thm:bc_diam2}}

In the following, we fix some small value $\eta > 0$: to fix ideas, we can take $\eta = \frac{1}{40}$. We provide a scenario under which two stages of freezing occur: a first stage in the near-critical window around $p_c$ (more precisely, in the time interval $[p_{-\lambda}(N),p_c]$, for some well-chosen $\lambda$ large enough), and then a second stage much later, at some time very close to $1$.

We use the construction depicted on Figure \ref{fig:two_chambers} to create two ``chambers'': as we will show, they have the property that each of them has a diameter $< N$, but the diameter of their union is $\geq N$. Note that the left and right parts in Figure \ref{fig:two_chambers} are each similar to the construction in Figure \ref{fig:one_chamber}.

\begin{figure}
\begin{center}

\includegraphics[width=.95\textwidth]{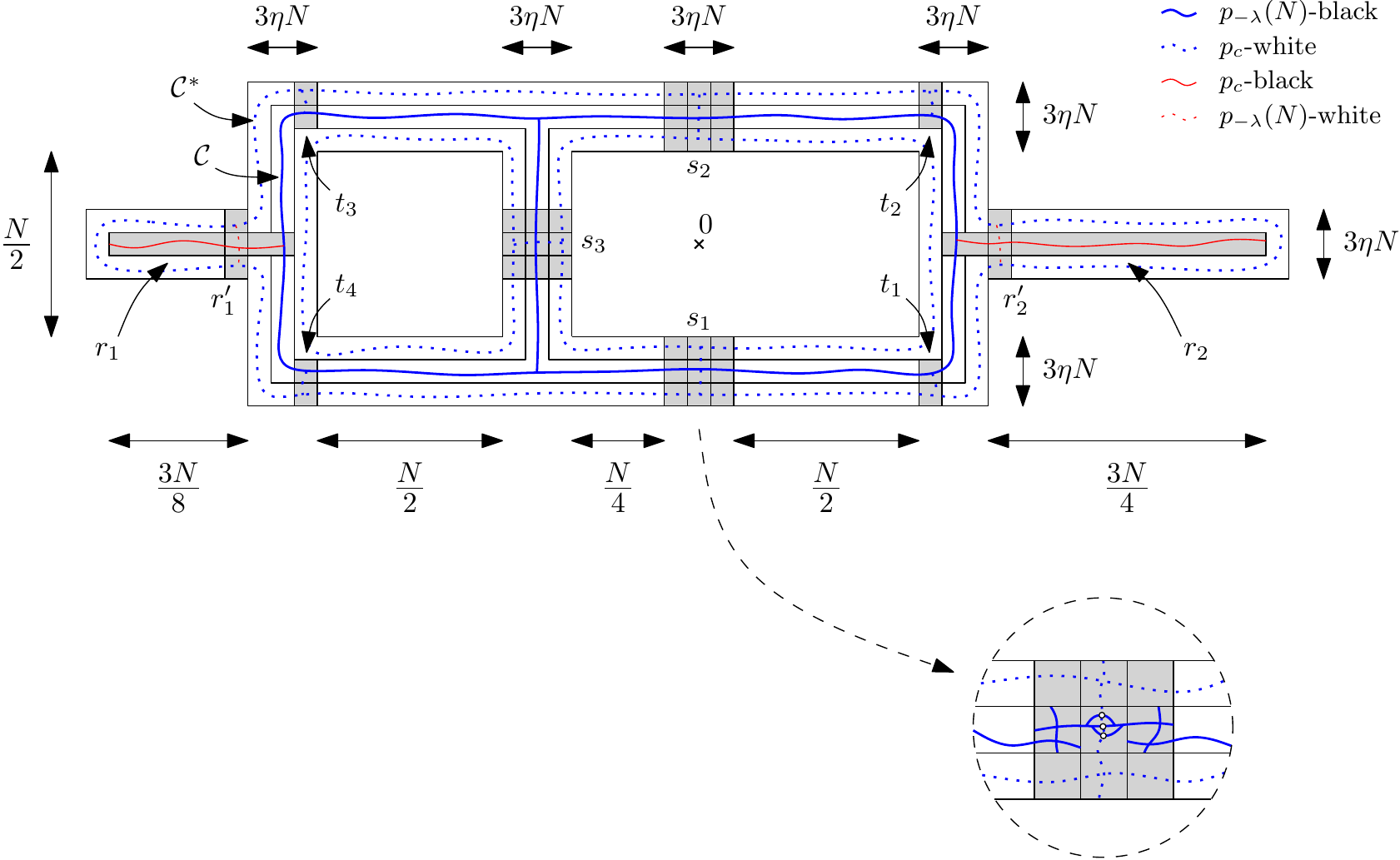}
\caption{\label{fig:two_chambers} Construction used to create two ``chambers'', with diameters (approximately) $\frac{N}{2}$ and $\frac{3N}{4}$. On this figure, the various corridors have width $\eta N$. In each of the three gray squares $s_i$ ($1 \leq i \leq 3$), we consider the event $\Gamma_{p_{-\lambda}(N)}^{\frac{\eta}{2} N,\delta}$ (properly translated, and also rotated by an angle $\frac{\pi}{2}$ in the case of $s_3$), where $\delta = \delta(\lambda) > 0$ is produced by Lemma \ref{lem:passage_sites}.}

\end{center}
\end{figure}

Let us fix $\ve > 0$ as in the statement of Theorem \ref{thm:bc_diam2}. For $\lambda \geq 0$, and $\delta = \delta(\lambda) > 0$ associated with $\lambda$ by Lemma \ref{lem:passage_sites}, we introduce the three events $\tilde{\Gamma}_i = \tilde{\Gamma}_i(N,\lambda,\ve)$ ($1 \leq i \leq 3$): $\tilde{\Gamma}_i$ is the event that in the square $s_i$, $\Gamma_{p_{-\lambda}(N)}^{\frac{\eta}{2} N,\delta}$ (translated, and rotated in the case of $s_3$) holds, and at least one of the (more than $\delta N^2 \pi_4(N)$) passage points in the ``inner square'' (with side length $\eta N$) is still white at time $1 - \frac{\ve}{N^2 \pi_4(N)}$ (i.e. at the beginning of the time interval in \eqref{eq:bc_diam2}).

By using Lemma \ref{lem:passage_sites}, and conditioning on the percolation configuration at time $p_c$, we obtain that for every $\lambda \geq 0$: for all $N \geq 1$,
$$\PP(\tilde{\Gamma}_i) \geq \delta' = \delta'(\lambda,\ve) > 0 \quad (1 \leq i \leq 3).$$

We denote by $\calE_N^{\lambda}$ the event (for the underlying percolation configuration) that the various paths depicted on Figure \ref{fig:two_chambers} exist, and the three events $\tilde{\Gamma}_i$ ($1 \leq i \leq 3$) occur. We first establish the following result.

\noindent \textbf{Claim:} By choosing $\lambda$ large enough, we can ensure that
\begin{equation}
\liminf_{N \to \infty} \PP \big( \calE_N^{\lambda} \big) > 0.
\end{equation}

\begin{proof}[Proof of the Claim]
The rectangles $r_1$ and $r_2$ have lengths $\big( \frac{3}{8} + 2\eta \big)N$ and $\big( \frac{3}{4} + 2\eta \big)N$, respectively, and both have width $\eta N$. Since they each have constant aspect ratio, it follows from RSW that
\begin{equation} \label{eq:panhandle1}
\PP_{p_c} \big( \Ch(r_1) \big) \geq \PP_{p_c} \big( \Ch(r_2) \big) \geq c_1,
\end{equation}
for some constant $c_1 = c_1(\eta) > 0$ independent of $N$ (recall that we consider $\eta$ to be fixed). We can then choose $\lambda > 0$ large enough so that, for all sufficiently large $N$,
\begin{equation} \label{eq:panhandle2}
\PP_{p_{-\lambda}(N)} \big( \Cv^*(r'_1) \big) = \PP_{p_{-\lambda}(N)} \big( \Cv^*(r'_2) \big) \geq 1 - \frac{c_1}{2}
\end{equation}
(by combining \eqref{eq:prop2_nc} with \eqref{eq:exp_decay}). In the remainder of the proof, we fix such a value $\lambda$, and we consider the constant $\delta = \delta(\lambda) > 0$ associated with it by Lemma \ref{lem:passage_sites}.

We now consider the event $\calF_N^{\lambda}$ that
\begin{itemize}
\item[(i)] all the $p_{-\lambda}(N)$-black and $p_c$-white paths on Figure \ref{fig:two_chambers}, except possibly the short vertical connections in $t_i$ ($1 \leq i \leq 4$), exist,

\item[(ii)] and the events $\tilde{\Gamma}_i$ ($1 \leq i \leq 3$) hold.
\end{itemize}
It follows from RSW \eqref{eq:RSW_nc} and Lemma \ref{lem:passage_sites} (recall our choice of $\delta$), which we can combine with the help of Lemma \ref{lem:Harris_gen}, that 
\begin{equation}
\PP \big( \calF_N^{\lambda} \big) \geq c_2
\end{equation}
for some $c_2 = c_2(\eta,\lambda) > 0$ independent of $N$.

\begin{figure}
\begin{center}

\includegraphics[width=.85\textwidth]{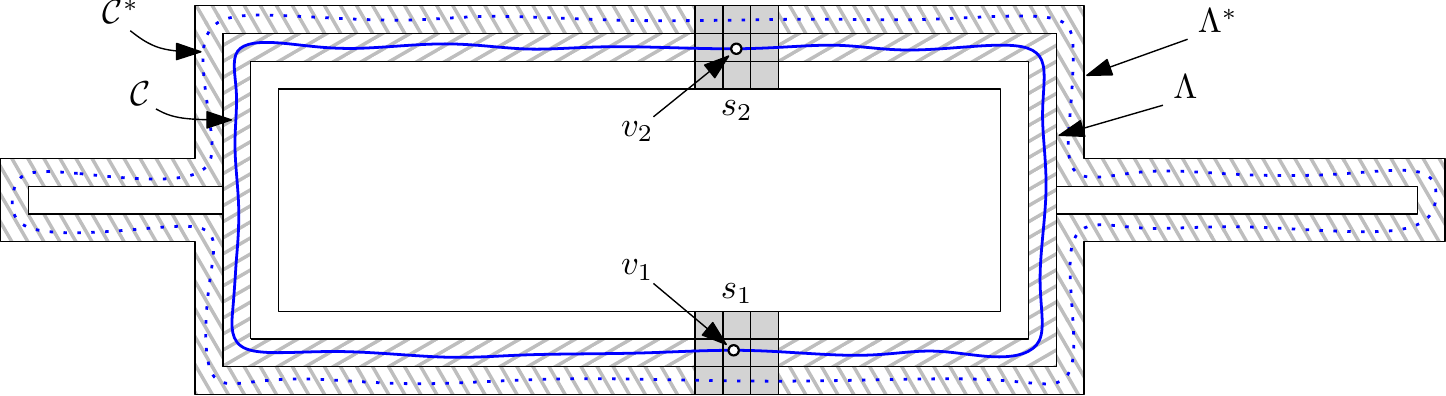}
\caption{\label{fig:two_chambers_annuli} This figure presents the ``annuli'' and the circuits involved in the proof of the Claim. $v_1$ and $v_2$ are two $p_c$-white ``passage sites'', in $s_1$ and $s_2$ respectively. We condition on the innermost circuit $\circuit$ in $\Lambda$, and the outermost circuit $\circuit^*$ in $\Lambda^*$, where $\circuit$ and $\circuit^*$ are as on the figure: $\circuit^*$ is $p_c$-white, and $\circuit$ is $p_{-\lambda}(N)$-black except at two sites, one in each of $s_1$ and $s_2$, which are $p_c$-white.}

\end{center}
\end{figure}

Using the notations of Figure \ref{fig:two_chambers_annuli}, we then condition on the innermost circuit $\circuit$ in $\Lambda$, and on the outermost circuit $\circuit^*$ in $\Lambda^*$, having the properties that $\circuit^*$ is $p_c$-white, and $\circuit$ is $p_{-\lambda}(N)$-black except on two $p_c$-white vertices $v_1 \in s_1$ and $v_2 \in s_2$. Now, consider the sites that lie between these two circuits \emph{and} outside of the squares $s_i$: the percolation configuration in this region is ``fresh''. We make the following observations.
\begin{itemize}
\item In each $t_i$ ($1 \leq i \leq 4$), there exists a vertical $p_c$-white connection between $\circuit$ and $\circuit^*$ with a probability $\geq c_3 > 0$, for some universal constant $c_3$ independent of $N$ (by RSW).
 
 \item The paths in $r_1$ and $r'_1$ (in red on Figure \ref{fig:two_chambers}), respectively $p_c$-black and $p_{-\lambda}(N)$-white, exist with a probability $\geq c_1 - \frac{c_1}{2} = \frac{c_1}{2}$ (by combining \eqref{eq:panhandle1} and \eqref{eq:panhandle2}).
 
 \item For the same reason, the red paths in $r_2$ and $r'_2$ exist with a probability $\geq \frac{c_1}{2}$.
\end{itemize}
Moreover, all these events are conditionally independent, so that the conditional probability of their intersection is at least $c_3^4 \big( \frac{c_1}{2} \big)^2$. We deduce
$$\PP \big( \calE_N^{\lambda} \big) \geq c_2 \cdot c_3^4 \big( \frac{c_1}{2} \big)^2 > 0,$$
which completes the proof of the Claim.
\end{proof}

We now assume that the event $\calE_N^{\lambda}$ holds, and we examine consequences of it for the modified diameter-frozen percolation process itself. First, note that all the $p_{-\lambda}(N)$-black and $p_c$-white paths in this event (in blue on Figure \ref{fig:two_chambers}) are present throughout the time interval $[p_{-\lambda}(N),p_c]$. The circuit $\circuit$ can be divided into two parts, to the left and to the right of the passage sites in $s_1$ and $s_2$, and we denote by $\cluster_L$ and $\cluster_R$ the connected components containing them.
\begin{itemize}
\item On the time interval $[p_{-\lambda}(N),p_c]$, there cannot be any other black connected component with diameter $\geq N$ inside $\circuit^*$, thanks to the $p_c$-white paths in the $t_i$'s ($1 \leq i \leq 4$).

\item $\cluster_L$ has a diameter $\leq 4 \eta N + \frac{N}{2} + 3 \eta N + \frac{N}{4} + 2 \eta N < N$ at time $p_{-\lambda}(N)$, because of the $p_{-\lambda}(N)$-white vertical path in $r'_1$ (here we also use our choice of $\eta$).

\item $\cluster_L$ has a diameter $\geq \frac{3N}{8} + \frac{N}{2} + \frac{N}{4} > N$ at time $p_c$ (using the $p_c$-black horizontal path in $r_1$). Hence, in the frozen percolation process, it freezes at some time in $[p_{-\lambda}(N),p_c]$.

\item Similarly, $\cluster_R$ has a diameter $\leq 2 \eta N + \frac{N}{2} + 4 \eta N < N$ at time $p_{-\lambda}(N)$, and $\geq \frac{N}{2} + \frac{3N}{4} > N$ at time $p_c$. Hence, in the frozen percolation process, it freezes at some time in $[p_{-\lambda}(N),p_c]$.
\end{itemize}
When these two clusters freeze, they create two chambers as desired, which are separated by a sequence of at least $\delta n^2 \pi_4(n)$ $p_c$-white ``passage sites'' in $s_3$.

Intuitively, it is now tempting to conclude the proof as follows. In the right chamber, at time close to $1$, there exists a net $\net$ with mesh $\frac{\eta}{4} N$ with very high probability (using Lemma \ref{lem:net}), so that any connected component with diameter $> \eta N$ inside this chamber has to intersect $\net$. Hence, $\net$ freezes, at the latest when it gets connected to the left chamber (but possibly earlier, due to connections to the outside through $s_1$ or $s_2$).


However, we have to take into account the possibility that after the time $T_3$ when all the passage sites in $s_3$ have become black, a large cluster with diameter $\geq N$ may occur without touching the net. Indeed, the boundary of the right chamber may go back and forth, thus leaving ``bubbles'' with large diameter, as shown on Figure \ref{fig:regularity_boundary} (left). In order to prevent this undesirable behavior, we introduce yet another event depicted on Figure \ref{fig:regularity_boundary} (right), ensuring some regularity for the boundary. With this additional event, the situation depicted on Figure \ref{fig:regularity_boundary} (left) cannot occur, and it is guaranteed that a cluster with diameter $\geq N$ emerging after time $T_3$, and containing passage sites in $s_3$, has to contain the net.

\begin{figure}
\begin{center}

\subfigure{\includegraphics[width=.35\textwidth]{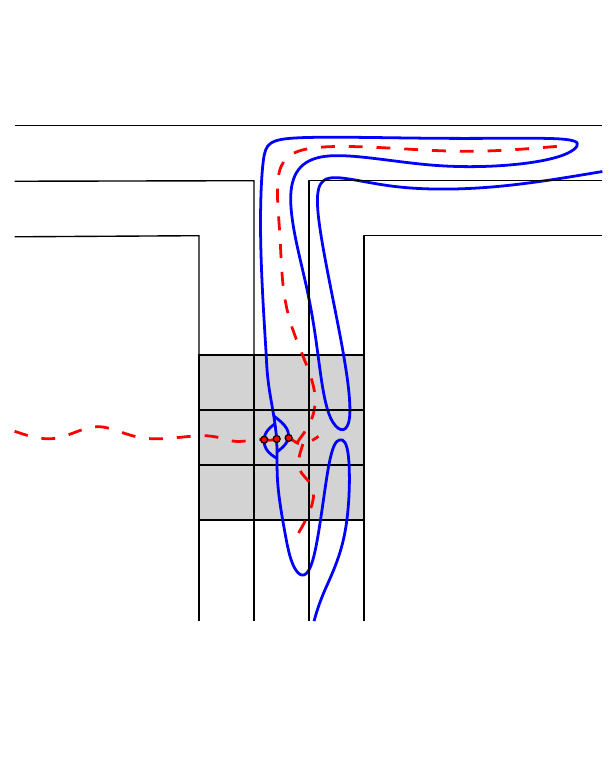}}
\hspace{0.12\textwidth}
\subfigure{\includegraphics[width=.4\textwidth]{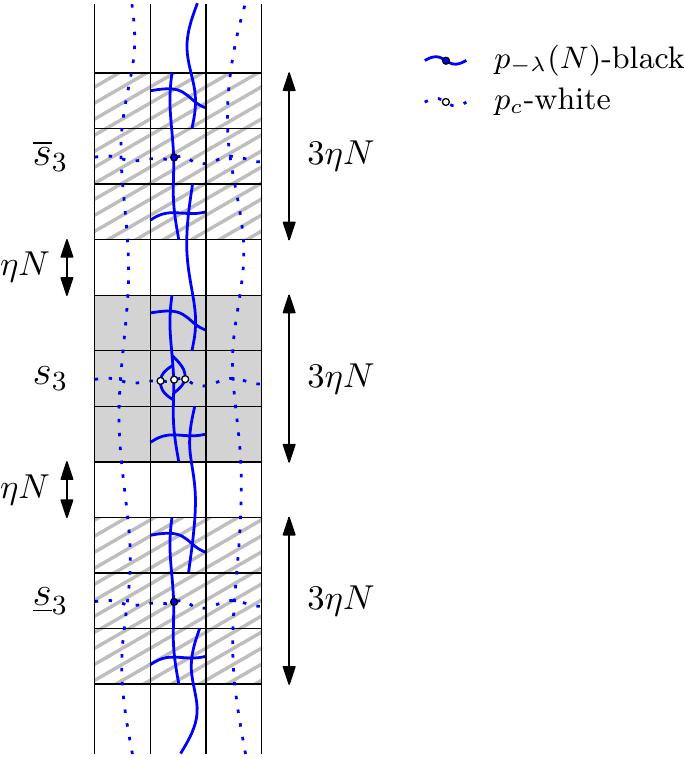}}\\

\caption{\label{fig:regularity_boundary} \emph{Left:} If the (frozen) boundary of the right chamber goes back and forth too much, then some connected component with a diameter larger than $N$ (in red dashed lines) may arise without intersecting the net inside the chamber. \emph{Right:} To circumvent this issue, we introduce two extra four-arm events in the boxes $\underline{s}_3$ and $\overline{s}_3$.}
\end{center}
\end{figure}

We are now in a position to conclude. Indeed, the extra cost of the configurations in $\underline{s}_3$ and $\overline{s}_3$ is just a positive uniform constant (this can be obtained in a similar way as Lemma \ref{lem:passage_sites}, but also in an elementary fashion by successive applications of RSW \eqref{eq:RSW} and conditionings). So, if we let $\tilde{\calE}_N^{\lambda}$ denote the event that $\calE_N^{\lambda}$ holds, as well as the additional event on Figure \ref{fig:regularity_boundary} (right) that we mentioned, we have
$$\PP \big( \tilde{\calE}_N^{\lambda} \big) \geq c_4 \cdot \PP \big( \calE_N^{\lambda} \big),$$
for some $c_4 = c_4(\lambda, \eta) > 0$.

For $\tilde{\ve} = \frac{\ve}{N^2 \pi_4(N)}$ (so that $1 - \tilde{\ve}$ is the beginning of the time interval in \eqref{eq:bc_diam2}), we introduce the following two events $E_i = E_i(\eta,N,\ve)$ ($i = 1, 2$).
\begin{itemize}
\item $E_1 := \net_{1 - \tilde{\ve}}(\frac{\eta}{4} N, 2N)$ (i.e. at time $1 - \tilde{\ve}$, there is a net with mesh $\frac{\eta}{4} N$ in the box $\Ball_{2N}$). Note that $\PP(E_1) \to 1$ as $N \to \infty$ (from Lemma \ref{lem:net}).

\item $E_2 := \{$there exists a $(1 - \tilde{\ve})$-black path from $0$ to $\partial \Ball_{2 \eta N}\}$. We have $\PP(E_2) \geq \theta(1 - \tilde{\ve}) \to 1$ as $N \to \infty$ (since $\tilde{\ve} \to 0$).
\end{itemize}
If $\tilde{\calE}_N^{\lambda}$, $E_1$ and $E_2$ occur, then the event in the left-hand side of \eqref{eq:bc_diam2} occurs as well. Hence, the latter event has probability at least
$$\PP \big( \tilde{\calE}_N^{\lambda} \cap E_1 \cap E_2\big) \geq \PP \big( \tilde{\calE}_N^{\lambda} \big) - \PP ( E_1^c ) - \PP ( E_2^c ),$$
which completes the proof.

\subsection{Remarks and open questions} \label{sec:diam_frozen_open}

As mentioned earlier, the boundary conditions are essential for the results of \cite{Ki2015}, which contrast with Theorem \ref{thm:bc_diam2}. It should be noted that the first important observation in this paper, namely that for every fixed $K > 1$, the number of frozen clusters in $\Ball_{KN}$ is tight in $N$ (Lemma 3.3 in \cite{Ki2015}) already breaks down for the modified model. Indeed, its proof makes use of the existence of white paths at some time $p = p_{-\lambda}(N)$, which are provided by the boundaries of frozen clusters. This observation is crucial for the arguments in \cite{Ki2015} since it then allows one to study the frozen clusters ``one by one'' (and ensures that the procedure ends after a finite number of steps).

Our results leave open the question whether clusters freeze at times bounded away from both $p_c$ and $1$. In particular, is it true that for every fixed finite connected $C$ with $0 \in C$,
$$\liminf_{N \to \infty} \PPdiam \big( \cluster_1(0) = C \big) > 0,$$
i.e. that microscopic clusters (with diameter of order $1$) have a positive density as well?

\section{Volume-frozen percolation} \label{sec:vol_frozen}

We now discuss briefly the influence of boundary conditions for volume-frozen percolation, and discuss how the results for the original process (where the boundary stays white forever) obtained in \cite{BN2015, BKN2015} could be adapted.

First, we observe that for all our arguments in \cite{BN2015}, we never use the particular definition of the boundary of a hole. The results in that paper are obtained by combining estimates on the volume of the largest black cluster in a box \cite{BCKS2001} with \eqref{eq:exp_decay} and \eqref{eq:equiv_theta} (the a-priori bounds \eqref{eq:RSW} and \eqref{eq:1arm} are also used). For the induction step, we use a geometric construction requiring the existence of black and white circuits in some prescribed annuli. These circuits ensure that when the giant black connected component freezes, it creates a hole whose boundary has to lie between these two circuits (and so in particular in the union of the two annuli). This observation holds true for both the original and the modified volume-frozen processes. Hence, the results from \cite{BN2015} about exceptional scales hold for the modified process as well: with the same sequence of exceptional scales $(m_k(N))_{k \geq 1}$, we observe the following dichotomy for the process in $\Ball_{m(N)}$.
\begin{itemize}
\item For all $c > 1$ and $k \geq 2$, if $m(N)$ satisfies $c^{-1} m_k(N) \leq m(N) \leq c m_k(N)$ for all $N$ large enough, then
$$\liminf_{N \to \infty} \PPvol (0 \text{ freezes for the process in } \Ball_{m(N)}) > 0.$$

\item For all $\ve > 0$ and $k \geq 1$, there exists $c > 1$ such that: if $c m_k(N) \leq m(N) \leq c^{-1} m_{k+1}(N)$ for all $N$ large enough, then
$$\limsup_{N \to \infty} \PPvol (0 \text{ freezes for the process in } \Ball_{m(N)}) \leq \ve.$$
\end{itemize}

On the other hand, the boundary conditions are used in a crucial way in \cite{BKN2015}, for the ``approximability'' and ``continuity'' lemmas (Lemmas 3.7 and 3.8, resp.), which require the existence of the right number of macroscopic arms (with some prescribed colors). We use in particular the a-priori upper bound on $6$-arm events \eqref{eq:46arms}. These lemmas are important to enable us to approximate the process by a Markov chain, using the separation of scales. The other arguments in \cite{BKN2015} do not depend on the particular boundary conditions.

%
%

\bibliographystyle{plain}
\bibliography{Modified_FP}

\end{document}